\documentclass[12pt]{article}
\usepackage{amsfonts,amsmath,amssymb}
\usepackage{mathrsfs,mathtools,url}
\usepackage[T1]{fontenc}
\usepackage{graphicx,tikz,color}
\usepackage{cite}

\newtheorem{theorem}{Theorem}[section]

\newtheorem{lemma}[theorem]{Lemma}
\newtheorem{proposition}[theorem]{Proposition}

\newtheorem{corollary}[theorem]{Corollary}

\newcommand{\gp}{{\rm gp}}
\newcommand{\gpt}{\gp_{\rm t}}
\newcommand{\gpd}{\gp_{\rm d}}
\newcommand{\gpo}{\gp_{\rm o}}
\newcommand{\mut}{\mu_{\rm t}}
\newcommand{\muo}{\mu_{\rm o}}
\newcommand{\mud}{\mu_{\rm d}}

\newcommand{\proof}{\noindent{\bf Proof.\ }}
\newcommand{\qed}{\hfill $\square$ \bigskip}

\title{Varieties of mutual-visibility and general position on Sierpi\'nski graphs}

\author{Dhanya Roy $^{a,}$\footnote{\tt dhanyaroyku@gmail.com, dhanyaroyku@cusat.ac.in} 
	\and Sandi Klav\v{z}ar $^{b,c,d,}$\footnote{\tt sandi.klavzar@fmf.uni-lj.si} 
	\and Aparna Lakshmanan S $^{a,}$\footnote{\tt aparnaren@gmail.com, aparnals@cusat.ac.in}
 \and Jing Tian $^{e,c,}$\footnote{\tt jingtian526@126.com} \\\\
	$^{a}$ \small Department of Mathematics, 	Cochin University of Science and Technology, 
	\\ \small Cochin - 22, Kerala, India\\
	$^{b}$ \small Faculty of Mathematics and Physics, University of Ljubljana, Slovenia\\
	$^{c}$ \small Institute of Mathematics, Physics and Mechanics, Ljubljana, Slovenia \\
	$^{d}$ \small Faculty of Natural Sciences and Mathematics, University of Maribor, Slovenia\\
   $^{e}$ \small School of Science, Zhejiang University of Science and Technology, \\
\small Hangzhou, Zhejiang 310023, PR China\\
}
\date{\today}

\begin{document}
	\maketitle
	
\begin{abstract}
The variety of mutual-visibility problems contains four members, as does the variety of general position problems. The basic problem is to determine the cardinality of the largest such sets. In this paper, these eight invariants are investigated on Sierpi\'nski graphs $S_p^n$. They are determined for the Sierpi\'nski graphs $S_p^2$, $p\ge 3$. All, but the outer mutual-visibility number and the outer general position number, are also determined for $S_3^n$, $n\ge 3$. In many of the cases the corresponding extremal sets are enumerated.
\end{abstract}
	
	
\medskip\noindent
{\bf AMS Subj.\ Class.\ (2020)}: 05C12, 05C69, 05C30

\section{Introduction}
	
General position and mutual-visibility are two fresh areas in metric and algorithmic graph theory. These concepts are complementary to each other, and together they represent a flourishing field of research. 

After general position sets were independently introduced (in a general setting) to graph theory in \cite{chandran-2016} and in~\cite{manuel-2018}, research in this area has expanded rapidly, a recent review article~\cite{chandran-2025} lists 115 references. These investigations include several interesting variations including edge general position sets~\cite{manuel-2022}, monophonic position sets~\cite{thomas-2024}, Steiner position sets~\cite{klavzar-2021}, vertex position sets~\cite{thankachy-2024}, mobile position sets~\cite{klavzar-2023}, and lower general position sets~\cite{distefano-2025, Kruft-2024}. See also recent studies~
\cite{araujo-2025, irsic-2024, thomas-2024a, tian-2024, tuite-2025, yao-2022}.

Given a set $X$ of vertices in a graph $G$, two vertices $u$ and $v$ are {\em $X$-positionable}, if \underline{for every shortest $u,v$-path} $P$ we have $V(P) \cap X \subseteq \{u,v\}$. (Note that if $uv\in E(G)$, then $u$ and $v$ are $X$-positionable.) Then $X$ is a {\em general position set}, if every $u, v \in X$ are $X$-positionable. A largest general position set is a {\em gp-set} and its size is the {\em general position number} $\gp(G)$ of $G$. 

Based on the motivation of robotic visibility, the graph mutual-visibility problem was introduced in 2022 by Di Stefano~\cite{distefano-2022}. Given a set $X$ of vertices in a graph $G$, two vertices $u$ and $v$ are {\em mutually-visible} with respect to $X$, shortly {\em $X$-visible}, if there \underline{exists a shortest $u,v$-path} $P$ such that $V(P) \cap X \subseteq \{u,v\}$. The set $X$ is a {\em mutual-visibility set} if any two vertices from $X$ are $X$-visible. A largest mutual-visibility set of $G$ is a {\em $\mu$-set} and its size is the {\em mutual-visibility number} $\mu(G)$ of $G$. Although only recently introduced, the mutual-visibility sets has already received a lot of attention, here we would like to point in particular to~\cite{BresarYero-2024, cicerone-2024b, Korze-2024, Kuziak-2023, Roy-2025, Tian-2024a}. 

In~\cite{cicerone-2024}, the total mutual-visibility number was introduced, while the variety of mutual-visibility invariants was rounded off in~\cite{cicerone-2023a} by adding to the list the outer mutual-visibility number and the dual mutual-visibility number. A set $X\subseteq V(G)$ is an {\em outer mutual-visibility set} in $G$ if $X$ is a mutual-visibility set and every pair of vertices $u \in X$, $v \in V(G)\setminus X$ are $X$-visible. $X$ is a {\em dual mutual-visibility set} if $X$ is a mutual-visibility set and every pair of vertices $u, v \in V(G)\setminus X$ are $X$-visible. Finally, $X$ is a {\em total mutual-visibility set} if every pair of vertices in $G$ are $X$-visible. The cardinality of a largest outer/dual/total mutual-visibility sets are respectively denoted by $\muo(G)$, $\mud(G)$, $\mut(G)$. 

Following the pattern of mutual-visibility, the variety of general position invariants was presented in~\cite{tian-2025}. The definition of the {\em outer/dual/total general position set} in $G$ is analogous, we just need to replace everywhere ``$X$-visible'' by ``$X$-positionable.'' Largest corresponding sets are called {\em $\gpo$-sets}, {\em $\gpd$-sets}, {\em $\gpt$-sets} and their sizes are the {\em outer/dual/total general position number} of $G$, respectively denoted by $\gpo(G)$, $\gpd(G)$, $\gpt(G)$.

Recently, Kor\v{z}e and Vesel~~\cite{Korze-2025+} investigated Sierpi\'nski triangle graphs $ST_3^n$ and determined $\tau(ST_3^n)$ for  $\tau\in \{\mu,\mut,\muo,\mud, \gp\}$. Sierpi\'nski triangle graphs $ST_3^n$ are obtained from the classical Sierpi\'nski graphs $S_3^n$ by contracting all the edges which do not lie in triangles. Continuing the above investigation, in this paper we determine $\tau(S_3^n)$ for  $\tau\in \{\mu,\mut, \mud, \gp, \gpt, \gpd\}$ and bound $\muo(S_3^n)$ and $\gpo(S_3^n)$. We also determine all the eight invariants for the Sierpi\'nski graphs $S_p^2$ for any $p\ge 3$. In many of the cases we also enumerate the corresponding extremal sets.

\section{Preliminaries}

For any positive integer $k$ we set $[k] = \{1,2,\ldots,k\}$ and $[k]_0 = \{0,1,\ldots,k-1\}$.

Let $G = (V(G), E(G))$ be a graph. The {\em degree} of a vertex $u$ of $G$ is the number of its adjacent vertices in $G$. For the vertices $u$ and $v$ of $G$, the length of a shortest $u,v$-path is called the {\em distance} between $u$ and $v$, and is denoted by $d_G(u,v)$. 
   
If $X\subseteq V(G)$, then the subgraph of $G$ induced by $X$ is denoted by $G[X]$. A vertex of a graph is {\em simplicial} if its neighborhood induces a complete graph. The set of simplicial vertices of $G$ is denoted by $S(G)$ and we set $s(G) = |S(G)|$. A subgraph $H$ of $G$ is {\em convex}, if for every two vertices $u$ and $v$ of $H$, every shortest $u,v$-path in $G$ is contained in $H$. 

Let $\tau\in \{\mu,\mut,\muo,\mud, \gp, \gpt,\gpo,\gpd\}$. By a {$\tau$-set} of $G$ we mean a set with the property $\tau$ of cardinality $\tau(G)$, and by {\em $\#\tau(G)$} we denote the number of $\tau$-sets of $G$. The following fact is often very useful, parts of it are already known from the literature. 

\begin{lemma}\label{lem:tau lower bound}
If $G$ is a connected graph and $\tau\in \{\mu,\mut,\muo,\mud, \gp,\gpo,\gpd\}$, then $\tau(G)\geq s(G)$.
\end{lemma}

\proof
Since any two vertices of $G$ are $S(G)$-positionable, $\gpt(G)\geq s(G)$. The assertion follows because $\gpt(G) \le \tau(G)$ for $\tau\in \{\mu,\mut,\muo,\mud, \gp, \gpt,\gpo,\gpd\}$. 
\qed

We next collect several known results that will be needed later. 

\begin{lemma}\label{lem:mu-convex}{\rm \cite[Lemma 2.1]{distefano-2022}}
If $H$ is a convex subgraph of $G$, and $X$ a mutual-visibility set of $G$, then $X\cap V(H)$ is a mutual-visibility set of $H$.
\end{lemma}

\begin{theorem} {\rm \cite[Theorem 5.2]{bujtas-2025+}}
\label{thm:distance two}
If $G$ is a connected graph and $X\subseteq V(G)$, then $X$ is a total mutual-visibility set of $G$ if and only if any two vertices $u$ and $v$ of $G$ with $d_G(u,v) = 2$ are $X$-visible.  
\end{theorem} 

\begin{theorem} {\rm \cite[Theorems~2.1, 3.1]{tian-2025}}
\label{thm:all}
If $G$ is a connected graph and $X\subseteq V(G)$, then the following hold.
\begin{enumerate}
\item[(i)] $X$ is a total general position set of $G$ if and only if $X\subseteq S(G)$. Moreover, $\gpt(G)=s(G)$.
\item[(ii)] If $X$ is a general position set of $G$, then $X$ is a dual general position set if and only if $G-X$ is convex.
\end{enumerate}
\end{theorem} 

In the rest of the preliminaries we introduced Sierpi\'nski graphs $S_p^n$ and related notation required. These graphs were introduced in~\cite{KlaMil-1997} as graphs of a particular variant of the well-known Tower of Hanoi problem~\cite{hinz-2018}. 

If $p \geq 3$ and $n \geq 1$, then $S_p^n$ is defined as follows. The vertex set is $V (S_p^n) = [p]_0^n$, we will simplify the notation of a vertex $(i_1, \ldots, i_n)$ of $S_p^n$ to $i_1\cdots i_n$. Vertices $i_1\cdots i_n$ and $j_1 \cdots j_n$ being adjacent if there exists an index $h  \in [n]$, such that
\begin{enumerate}
\item[(i)] $\forall$ $t$, $t < h$ $\implies$ $i_t = j_t$, 
\item[(ii)] $i_h \neq j_h$,
\item[(iii)] $\forall$ $t$, $t > h$ $\implies$ $i_t = j_h$ and $j_t = i_h$.
\end{enumerate}
In the case $p=3$, these graphs are isomorphic to the graphs of the classical Tower of Hanoi problem. 

If $s \in [p]_0^{n-k}$, where $k \in [n-1]$, then the subgraph of $S_p^n$ induced by the vertices of the form $\{st:\ t \in [p]_0^k\}$, is isomorphic to $S_p^k$, it will be denoted by $\underline{s}S_p^{k}$.  If $i\in [p]_0$, then the notation $\underline{i}S_p^1$ will be simplified to $iS_p^1$. Note that $iS_p^1$ is isomorphic to $K_p$. The subgraphs $\underline{s}S_p^{k}$ indicate the fractal nature of Sierpi\'nski graphs by which we mean that $V(S_p^n)$ can be partitioned into $p^{n-k}$ sets each of which induces a subgraph isomorphic to $S_p^{k}$.
 
 Now, consider $p=3$. Let $k$ be a positive integer, where $1\le k\le n-2$, and let $s \in [3]_0^k$. In the subgraph $\underline{s}S_3^{n-k}$ of $S_3^n$, each of the three vertices $si^{n-k}$, $i \in [3]_0$, is the degree $2$ vertex of an induced bull, which we denote by $\underline{s}B_i^n$. (Recall that the {\em bull graph} is a graph of order five obtained from a triangle by attaching two pendant vertices to its two different vertices.) Note that some of these bulls can be isomorphic. In particular, for a fixed $i\in [3]_0$, the bulls $\underline{i^k}B_i^{n}$, $1\le k\le n-2$, are one and the same bull with the degree two vertex being the vertex $i^n$. See Fig.~\ref{fig:S_3^4} where $S_3^4$ is presented and some of its bulls emphasized.  We can infer that $$V(\underline{s}B_i^n) = \{si^{n-k-2}ji,si^{n-k-1}j:\ j \in [3]_0\}\,.$$

 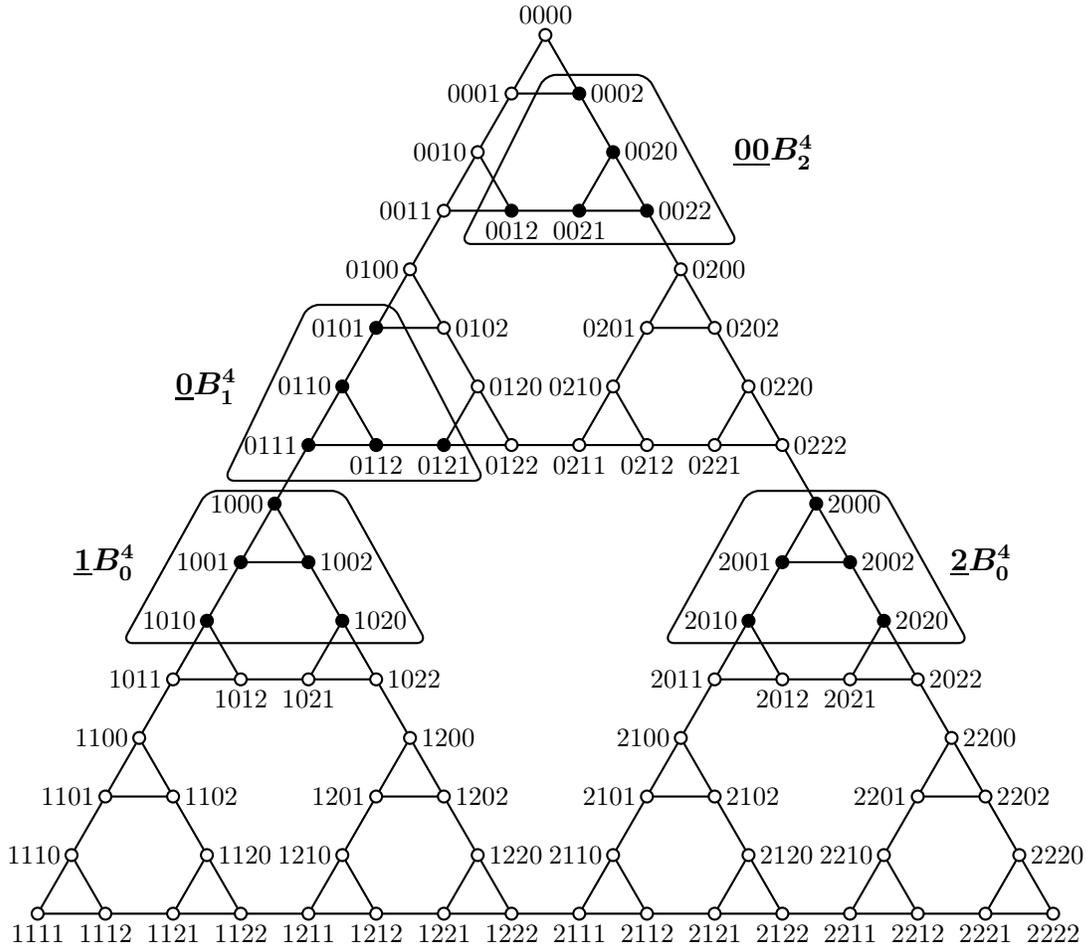
\begin{figure}[ht!]
\begin{center}
\begin{tikzpicture}[scale=0.45,style=thick,x=1cm,y=1cm]
\def\vr{5pt}
\begin{scope}[xshift=10cm, yshift=10cm] 

\coordinate(x1) at (0,0);
\coordinate(x2) at (1,1.732);
\coordinate(x3) at (2,3.464);
\coordinate(x4) at (3,5.196);
\coordinate(x5) at (4,3.464);
\coordinate(x6) at (5,1.732);
\coordinate(x7) at (6,0);
\coordinate(x8) at (4,0);
\coordinate(x9) at (2,0);

\coordinate(x10) at (8,0);
\coordinate(x11) at (9,1.732);
\coordinate(x12) at (10,3.464);
\coordinate(x13) at (11,5.196);
\coordinate(x14) at (12,3.464);
\coordinate(x15) at (13,1.732);
\coordinate(x16) at (14,0);
\coordinate(x17) at (12,0);
\coordinate(x18) at (10,0);

\coordinate(x19) at (4,6.928);
\coordinate(x20) at (5,8.66);
\coordinate(x21) at (6,10.392);
\coordinate(x22) at (7,12.124);
\coordinate(x23) at (8,10.392);
\coordinate(x24) at (9,8.66);
\coordinate(x25) at (10,6.928);
\coordinate(x26) at (8,6.928);
\coordinate(x27) at (6,6.928);


\coordinate(x28) at (16,0);
\coordinate(x29) at (17,1.732);
\coordinate(x30) at (18,3.464);
\coordinate(x31) at (19,5.196);
\coordinate(x32) at (20,3.464);
\coordinate(x33) at (21,1.732);
\coordinate(x34) at (22,0);
\coordinate(x35) at (20,0);
\coordinate(x36) at (18,0);

\coordinate(x37) at (24,0);
\coordinate(x38) at (25,1.732);
\coordinate(x39) at (26,3.464);
\coordinate(x40) at (27,5.196);
\coordinate(x41) at (28,3.464);
\coordinate(x42) at (29,1.732);
\coordinate(x43) at (30,0);
\coordinate(x44) at (28,0);
\coordinate(x45) at (26,0);

\coordinate(x46) at (20,6.928);
\coordinate(x47) at (21,8.66);
\coordinate(x48) at (22,10.392);
\coordinate(x49) at (23,12.124);
\coordinate(x50) at (24,10.392);
\coordinate(x51) at (25,8.66);
\coordinate(x52) at (26,6.928);
\coordinate(x53) at (24,6.928);
\coordinate(x54) at (22,6.928);

\coordinate(x55) at (8,13.856);
\coordinate(x56) at (9,15.588);
\coordinate(x57) at (10,17.32);
\coordinate(x58) at (11,19.052);
\coordinate(x59) at (12,17.32);
\coordinate(x60) at (13,15.588);
\coordinate(x61) at (14,13.856);
\coordinate(x62) at (12,13.856);
\coordinate(x63) at (10,13.856);

\coordinate(x64) at (16,13.856);
\coordinate(x65) at (17,15.588);
\coordinate(x66) at (18,17.32);
\coordinate(x67) at (19,19.052);
\coordinate(x68) at (20,17.32);
\coordinate(x69) at (21,15.588);
\coordinate(x70) at (22,13.856);
\coordinate(x71) at (20,13.856);
\coordinate(x72) at (18,13.856);

\coordinate(x73) at (12,20.784);
\coordinate(x74) at (13,22.516);
\coordinate(x75) at (14,24.248);
\coordinate(x76) at (15,25.98);
\coordinate(x77) at (16,24.248);
\coordinate(x78) at (17,22.516);
\coordinate(x79) at (18,20.784);
\coordinate(x80) at (16,20.784);
\coordinate(x81) at (14,20.784);

\draw (x1) -- (x2) -- (x3) -- (x4) -- (x5) -- (x6)--(x7)--(x8)--(x9)--(x1);
\draw (x2) -- (x9);
\draw (x3) -- (x5);
\draw (x6) -- (x8);

\draw (x10) -- (x11) -- (x12) -- (x13) -- (x14) -- (x15)--(x16)--(x17)--(x18)--(x10);
\draw (x11) -- (x18);
\draw (x12) -- (x14);
\draw (x15) -- (x17);

\draw (x19) -- (x20) -- (x21) -- (x22) -- (x23) -- (x24)--(x25)--(x26)--(x27)--(x19);
\draw (x20) -- (x27);
\draw (x21) -- (x23);
\draw (x24) -- (x26);

\draw (x4) -- (x19);
\draw (x25) -- (x13);
\draw (x7) -- (x10);

\draw (x28) -- (x29) -- (x30) -- (x31) -- (x32) -- (x33)--(x34)--(x35)--(x36)--(x28);
\draw (x29) -- (x36);
\draw (x30) -- (x32);
\draw (x33) -- (x35);

\draw (x37) -- (x38) -- (x39) -- (x40) -- (x41) -- (x42)--(x43)--(x44)--(x45)--(x37);
\draw (x38) -- (x45);
\draw (x39) -- (x41);
\draw (x42) -- (x44);

\draw (x46) -- (x47) -- (x48) -- (x49) -- (x50) -- (x51)--(x52)--(x53)--(x54)--(x46);
\draw (x47) -- (x54);
\draw (x48) -- (x50);
\draw (x51) -- (x53);

\draw (x31) -- (x46);
\draw (x52) -- (x40);
\draw (x34) -- (x37);

\draw (x55) -- (x56) -- (x57) -- (x58) -- (x59) -- (x60)--(x61)--(x62)--(x63)--(x55);
\draw (x56) -- (x63);
\draw (x57) -- (x59);
\draw (x60) -- (x62);

\draw (x64) -- (x65) -- (x66) -- (x67) -- (x68) -- (x69)--(x70)--(x71)--(x72)--(x64);
\draw (x65) -- (x72);
\draw (x66) -- (x68);
\draw (x69) -- (x71);

\draw (x73) -- (x74) -- (x75) -- (x76) -- (x77) -- (x78)--(x79)--(x80)--(x81)--(x73);
\draw (x74) -- (x81);
\draw (x75) -- (x77);
\draw (x78) -- (x80);

\draw (x58) -- (x73);
\draw (x79) -- (x67);
\draw (x61) -- (x64);

\draw (x22) -- (x55);
\draw (x70) -- (x49);
\draw (x16) -- (x28);

\draw(x1)[fill=white] circle(\vr) node[below]{{\footnotesize 1111}};
\draw(x2)[fill=white] circle(\vr) node[left]{{\footnotesize 1110}};
\draw(x3)[fill=white] circle(\vr) node[left]{{\footnotesize 1101}};
\draw(x4)[fill=white] circle(\vr) node[left]{{\footnotesize 1100}};
\draw(x5)[fill=white] circle(\vr) node[right]{{\footnotesize 1102}};
\draw(x6)[fill=white] circle(\vr) node[right]{{\footnotesize 1120}};
\draw(x7)[fill=white] circle(\vr) node[below]{{\footnotesize 1122}};
\draw(x8)[fill=white] circle(\vr) node[below]{{\footnotesize 1121}};
\draw(x9)[fill=white] circle(\vr) node[below]{{\footnotesize 1112}};

\draw(x10)[fill=white] circle(\vr) node[below]{{\footnotesize 1211}};
\draw(x11)[fill=white] circle(\vr) node[left]{{\footnotesize 1210}};
\draw(x12)[fill=white] circle(\vr) node[left]{{\footnotesize 1201}};
\draw(x13)[fill=white] circle(\vr) node[right]{{\footnotesize 1200}};
\draw(x14)[fill=white] circle(\vr) node[right]{{\footnotesize 1202}};
\draw(x15)[fill=white] circle(\vr) node[right]{{\footnotesize 1220}};
\draw(x16)[fill=white] circle(\vr) node[below]{{\footnotesize 1222}};
\draw(x17)[fill=white] circle(\vr) node[below]{{\footnotesize 1221}};
\draw(x18)[fill=white] circle(\vr) node[below]{{\footnotesize 1212}};

\draw(x19)[fill=white] circle(\vr) node[left]{{\footnotesize 1011}};
\draw(x20)[fill=black] circle(\vr) node[left]{{\footnotesize 1010}};
\draw(x21)[fill=black] circle(\vr) node[left]{$\boldsymbol{\underline{1}B_0^4}$ \,\,\,\, {\footnotesize 1001}};
\draw(x22)[fill=black] circle(\vr) node[left]{{\footnotesize 1000}};
\draw(x23)[fill=black] circle(\vr) node[right]{{\footnotesize 1002}};
\draw(x24)[fill=black] circle(\vr) node[right]{{\footnotesize 1020}};
\draw(x25)[fill=white] circle(\vr) node[right]{{\footnotesize 1022}};
\draw(x26)[fill=white] circle(\vr) node[below]{{\footnotesize 1021}};
\draw(x27)[fill=white] circle(\vr) node[below]{{\footnotesize 1012}};

\draw(x28)[fill=white] circle(\vr) node[below]{{\footnotesize 2111}};
\draw(x29)[fill=white] circle(\vr) node[left]{{\footnotesize 2110}};
\draw(x30)[fill=white] circle(\vr) node[left]{{\footnotesize 2101}};
\draw(x31)[fill=white] circle(\vr) node[left]{{\footnotesize 2100}};
\draw(x32)[fill=white] circle(\vr) node[right]{{\footnotesize 2102}};
\draw(x33)[fill=white] circle(\vr) node[right]{{\footnotesize 2120}};
\draw(x34)[fill=white] circle(\vr) node[below]{{\footnotesize 2122}};
\draw(x35)[fill=white] circle(\vr) node[below]{{\footnotesize 2121}};
\draw(x36)[fill=white] circle(\vr) node[below]{{\footnotesize 2112}};

\draw(x37)[fill=white] circle(\vr) node[below]{{\footnotesize 2211}};
\draw(x38)[fill=white] circle(\vr) node[left]{{\footnotesize 2210}};
\draw(x39)[fill=white] circle(\vr) node[left]{{\footnotesize 2201}};
\draw(x40)[fill=white] circle(\vr) node[right]{{\footnotesize 2200}};
\draw(x41)[fill=white] circle(\vr) node[right]{{\footnotesize 2202}};
\draw(x42)[fill=white] circle(\vr) node[right]{{\footnotesize 2220}};
\draw(x43)[fill=white] circle(\vr) node[below]{{\footnotesize 2222}};
\draw(x44)[fill=white] circle(\vr) node[below]{{\footnotesize 2221}};
\draw(x45)[fill=white] circle(\vr) node[below]{{\footnotesize 2212}};

\draw(x46)[fill=white] circle(\vr) node[left]{{\footnotesize 2011}};
\draw(x47)[fill=black] circle(\vr) node[left]{{\footnotesize 2010}};
\draw(x48)[fill=black] circle(\vr) node[left]{{\footnotesize 2001}};
\draw(x49)[fill=black] circle(\vr) node[right]{{\footnotesize 2000}};
\draw(x50)[fill=black] circle(\vr) node[right]{{\footnotesize 2002} \,\,\, $\boldsymbol{\underline{2}B_0^4}$};
\draw(x51)[fill=black] circle(\vr) node[right]{{\footnotesize 2020}};
\draw(x52)[fill=white] circle(\vr) node[right]{{\footnotesize 2022}};
\draw(x53)[fill=white] circle(\vr) node[below]{{\footnotesize 2021}};
\draw(x54)[fill=white] circle(\vr) node[below]{{\footnotesize 2012}};

\draw(x55)[fill=black] circle(\vr) node[left]{ {\footnotesize 0111}};
\draw(x56)[fill=black] circle(\vr) node[left]{$\boldsymbol{\underline{0}B_1^4}$ \,\,\,\, {\footnotesize 0110}};
\draw(x57)[fill=black] circle(\vr) node[left]{{\footnotesize 0101}};
\draw(x58)[fill=white] circle(\vr) node[left]{{\footnotesize 0100}};
\draw(x59)[fill=white] circle(\vr) node[right]{{\footnotesize 0102}};
\draw(x60)[fill=white] circle(\vr) node[right]{{\footnotesize 0120}};
\draw(x61)[fill=white] circle(\vr) node[below]{{\footnotesize 0122}};
\draw(x62)[fill=black] circle(\vr) node[below]{{\footnotesize 0121}};
\draw(x63)[fill=black] circle(\vr) node[below]{{\footnotesize 0112}};

\draw(x64)[fill=white] circle(\vr) node[below]{{\footnotesize 0211}};
\draw(x65)[fill=white] circle(\vr) node[left]{{\footnotesize 0210}};
\draw(x66)[fill=white] circle(\vr) node[left]{{\footnotesize 0201}};
\draw(x67)[fill=white] circle(\vr) node[right]{{\footnotesize 0200}};
\draw(x68)[fill=white] circle(\vr) node[right]{{\footnotesize 0202}};
\draw(x69)[fill=white] circle(\vr) node[right]{{\footnotesize 0220}};
\draw(x70)[fill=white] circle(\vr) node[right]{{\footnotesize 0222}};
\draw(x71)[fill=white] circle(\vr) node[below]{{\footnotesize 0221}};
\draw(x72)[fill=white] circle(\vr) node[below]{{\footnotesize 0212}};

\draw(x73)[fill=white] circle(\vr) node[left]{{\footnotesize 0011}};
\draw(x74)[fill=white] circle(\vr) node[left]{{\footnotesize 0010}};
\draw(x75)[fill=white] circle(\vr) node[left]{{\footnotesize 0001}};
\draw(x76)[fill=white] circle(\vr) node[above]{{\footnotesize 0000}};
\draw(x77)[fill=black] circle(\vr) node[right]{{\footnotesize 0002}};
\draw(x78)[fill=black] circle(\vr) node[right]{{\footnotesize 0020} \, \, \, $\boldsymbol{\underline{00}B_2^4}$};
\draw(x79)[fill=black] circle(\vr) node[right]{{\footnotesize 0022}};
\draw(x80)[fill=black] circle(\vr) node[below]{{\footnotesize 0021}};
\draw(x81)[fill=black] circle(\vr) node[below]{{\footnotesize 0012}};

\draw[rounded corners](2.5,8)--(5.0,12.5)--(9.0,12.5)--(11.5,8)--cycle;
\draw[rounded corners](18.5,8)--(21.0,12.5)--(25.0,12.5)--(27.5,8)--cycle;
\draw[rounded corners](5.5,12.8)--(8.0,18.0)--(10.5,18.0)--(13.2,12.8)--cycle;
\draw[rounded corners](12.5,19.8)--(15.0,24.8)--(18.0,24.8)--(20.7,19.8)--cycle;


\end{scope}
\end{tikzpicture}
\caption{$S_3^4$ and some of its bulls}
\label{fig:S_3^4}
\end{center}
\end{figure}

\section{Sierpi\'nski graphs $S_p^2$}
\label{sec:Sp2}

In the seminal paper on Sierpi\'nski graphs~\cite{KlaMil-1997} it was proved that there are at most two shortest paths between any two vertices of $S_p^n$. It was also described when one of the two cases happens. In particular, in $S_p^2$ there exist two shortest paths between any pair of vertices of the form $ik$ and $jk$, these are the paths $ik,ij, ji, jk$ and $ik, ki, kj, jk$. For all the remaining pair of vertices there exists a unique shortest path between them. 

\begin{theorem}
\label{thm:mu-Sp2}
If $p\geq 3$ then, 
$$
\mu(S_p^2) = \begin{cases}
\frac{(p+1)^2}{4}; & p \text{ odd}\,,\\[5pt]
\frac{p(p+2)}{4}; & p \text{ even}\,,
\end{cases} 
\quad \text{and} \quad 
\# \mu(S_p^2) = \begin{cases}
\binom{p}{\frac{p+1}{2}}; & p \text{ odd}\,,\\[5pt]
\binom{p+1}{\frac{p+2}{2}}; & p \text{ even}\,.
\end{cases}
$$
\end{theorem}

\proof
If $p\geq 3$ is odd, then let 
$$X_1 = \Big\{ii, ij:\ i\in [(p+1)/2]_0, j \in [p]_0\setminus [(p+1)/2]_0\Big\}\,,$$
and if $p\geq 4$ is even, then let 
$$X_2 = \Big\{ii, ij:\ i\in [p/2]_0, j \in [p]_0\setminus [p/2]_0\Big\}\,.$$
It is straightforward to check that $X_1$ is a mutual-visibility set of $S_p^2$ if $p$ is odd, whilst $X_2$ is  a mutual-visibility set of $S_p^2$ if $p$ is even. Since $|X_1| = (p+1)^2/4$ and $|X_2| = p(p+2)/4$, we have thus shown that 
     $$\mu(S_p^2)\geq \begin{cases}
\frac{(p+1)^2}{4}; & p \text{ odd}\,,\\[5pt]
\frac{p(p+2)}{4}; & p \text{ even}\,.
\end{cases} $$

To prove that this lower bound is also an upper bound, consider an arbitrary $\mu$-set $X$ of $S_p^2$. We may without loss of generality assume that 
$$|X\cap V(0S_p^1)| = \max \{|X\cap V(iS_p^1)|:\ i\in [p]_0\}\,,$$ 
and let $|X\cap V(0S_p^1)| = k$. Since we have assumed that $X$ is a $\mu$-set of $S_p^2$, the already proved lower bound implies that $k\geq 2$. We consider the following two cases. 

\medskip\noindent
{\bf Case 1}: $00\in X\cap V(0S_p^1)$.\\
Let $0j\in X\cap V(0S_p^1)$, where $j\in [p-1]$. Since $00, 0j, j0, ji$, where $i\in [p]_0$, is the unique shortest path between $00$ and $ji$, it follows that $X\cap V(jS_p^1) = \emptyset$. As $X\cap V(0S_p^1)$ contains $k-1$ vertices different from $00$, this in turn implies that $k-1$ subgraphs of the form $iS_p^1$ contain no vertex from $X$. By the definition of $k$ we get that $|X| \le k\cdot (p-k+1)$. 

\medskip\noindent
{\bf Case 2}: $00\notin X\cap V(0S_p^1)$.\\
Let $0j, 0j'\in X\cap V(0S_p^1)$, where $j\ne j'$ and $j,j'\in[p-1]$.
We claim that either $X\cap V(jS_p^1) = \emptyset$ or $X\cap V(j'S_p^1) = \emptyset$. Suppose not. Since $0j', 0j, j0, j\ell$ is the unique shortest path between $0j'$ and $j\ell$, where $\ell\in [p]_0\setminus\{j'\}$, it follows that $X\cap V(jS_p^1) = \{jj'\}$. Analogously, since $0j, 0j', j'0, j'\ell'$ is the unique shortest between $0j$ and $j'\ell'$, where $\ell'\in [p]_0\setminus \{j\}$, we get $X\cap V(j'S_p^1) = \{j'j\}$. Hence 
$\{0j,0j',jj',j'j\}\subseteq X$. But the vertices $0j$, $0j'$, $j0$, $jj'$, $j'j$, and $j'0$ induce a cycle $C_6$, it contradicts with the fact that $\mu(C_6) = 3$. Since $X\cap V(0S_p^1)$ contains $k-1$ vertices different from $0j$, it implies that $k-1$ subgraphs of the form $iS_p^1$ contain no vertex from $X$. By the definition of $k$ we thus have $|X| \le k\cdot (p-k+1)$. 

From the above, we have 
$$|X|\leq  k\cdot (p-k+1) \leq \max\{k\cdot (p-k+1):\ k\in [p]\}\,.$$
Note that
$$
\max\{k\cdot (p-k+1):\ k\in [p]\} = \begin{cases}
\frac{(p+1)^2}{4}; & p \text{ odd}\,,\\[5pt]
\frac{p(p+2)}{4}; & p \text{ even}\,.
\end{cases} 
$$
As a consequence, we conclude that
$$
\mu(S_p^2) = \begin{cases}
\frac{(p+1)^2}{4}; & p \text{ odd}\,,\\[5pt]
\frac{p(p+2)}{4}; & p \text{ even}\,.
\end{cases}
$$

In remains to determine the number of $\mu$-sets $X$. Assume first that $p$ is odd.  In this case $|X| = \frac{(p+1)^2}{4}$, which is if and only if $k = \frac{p+1}{2}$. That is, $X$ intersects exactly $\frac{p+1}{2}$ subgraphs $iS_p^1$ in exactly $\frac{p+1}{2}$ vertices each. The selection of these subgraphs can be made in $\binom{p}{\frac{p+1}{2}}$ ways. We now claim that as soon as such a selection is made, $X$ is uniquely determined. To prove it, assume without loss of generality that $X$ has vertices in $iS_p^2$ for $i\in [(p+1)/2]_0$. Hence, if $j, k\in [(p+1)/2]_0$, then $jk\notin X$. The remaining vertices in each of $iS_p^2$ for $i\in [(p+1)/2]_0$ must thus lie in $X$, that is, $X$ is uniquely determined. This proves that $\# \mu(S_p^2) = \binom{p}{\frac{p+1}{2}}$ when $p$ is odd. 

The argument for $p$ is even is similar, except that now a $\mu$-set either intersects $\frac{p}{2}$ copies $iS_p^1$ in exactly $\frac{p+2}{2}$ vertices each, or intersects $\frac{p+2}{2}$ copies of $iS_p^1$ in exactly $\frac{p}{2}$ vertices each. In each of these cases we then proceed as above to see that a $\mu$-set is unique as soon as we select the subgraphs $iS_p^1$ which contain vertices from the $\mu$-set. Therefore if $p$ is even,
$$
\# \mu(S_p^2) = \binom{p}{\frac{p}{2}} + \binom{p}{\frac{p+2}{2}} = \binom{p+1}{\frac{p+2}{2}}\,,
$$
and we are done. 
\qed

Theorem~\ref{thm:mu-Sp2} is illustrated in Fig.~\ref{fig:mu-S32-S42} for $p\in \{3, 4\}$. For $S_3^2$ all three $\mu$-sets are shown, while for $S_4^2$ the left figure shows one of six $\mu$-sets that interest two subgraphs $iS_4^1$, and the right figure shows one of four $\mu$-sets that interest three subgraphs $iS_4^1$. 

\begin{figure}[ht!]
\begin{center}
\begin{tikzpicture}[scale=0.5,style=thick,x=1cm,y=1cm]
\def\vr{5pt}

\begin{scope}[xshift=0cm, yshift=8cm] 
\coordinate(x1) at (0,0);
\coordinate(x2) at (1,1.732);
\coordinate(x3) at (2,3.461);
\coordinate(x4) at (3,5.196);
\coordinate(x5) at (4,3.461);
\coordinate(x6) at (5,1.732);
\coordinate(x7) at (6,0);
\coordinate(x8) at (4,0);
\coordinate(x9) at (2,0);

\draw (x1) -- (x2) -- (x3) -- (x4) -- (x5) -- (x6)--(x7)--(x8)--(x9)--(x1);
\draw (x2) -- (x9);
\draw (x3) -- (x5);
\draw (x6) -- (x8);

\draw(x1)[fill=white] circle(\vr);
\draw(x2)[fill=white] circle(\vr);
\draw(x3)[fill=black] circle(\vr);
\draw(x4)[fill=black] circle(\vr);
\draw(x5)[fill=white] circle(\vr);
\draw(x6)[fill=white] circle(\vr);
\draw(x7)[fill=black] circle(\vr);
\draw(x8)[fill=black] circle(\vr);
\draw(x9)[fill=white] circle(\vr);
\end{scope}


\begin{scope}[xshift=7cm, yshift=8cm] 
\coordinate(x1) at (0,0);
\coordinate(x2) at (1,1.732);
\coordinate(x3) at (2,3.461);
\coordinate(x4) at (3,5.196);
\coordinate(x5) at (4,3.461);
\coordinate(x6) at (5,1.732);
\coordinate(x7) at (6,0);
\coordinate(x8) at (4,0);
\coordinate(x9) at (2,0);

\draw (x1) -- (x2) -- (x3) -- (x4) -- (x5) -- (x6)--(x7)--(x8)--(x9)--(x1);
\draw (x2) -- (x9);
\draw (x3) -- (x5);
\draw (x6) -- (x8);

\draw(x1)[fill=black] circle(\vr);
\draw(x2)[fill=white] circle(\vr);
\draw(x3)[fill=white] circle(\vr);
\draw(x4)[fill=black] circle(\vr);
\draw(x5)[fill=black] circle(\vr);
\draw(x6)[fill=white] circle(\vr);
\draw(x7)[fill=white] circle(\vr);
\draw(x8)[fill=white] circle(\vr);
\draw(x9)[fill=black] circle(\vr);
\end{scope}


\begin{scope}[xshift=14cm, yshift=8cm] 
\coordinate(x1) at (0,0);
\coordinate(x2) at (1,1.732);
\coordinate(x3) at (2,3.461);
\coordinate(x4) at (3,5.196);
\coordinate(x5) at (4,3.461);
\coordinate(x6) at (5,1.732);
\coordinate(x7) at (6,0);
\coordinate(x8) at (4,0);
\coordinate(x9) at (2,0);

\draw (x1) -- (x2) -- (x3) -- (x4) -- (x5) -- (x6)--(x7)--(x8)--(x9)--(x1);
\draw (x2) -- (x9);
\draw (x3) -- (x5);
\draw (x6) -- (x8);

\draw(x1)[fill=black] circle(\vr);
\draw(x2)[fill=black] circle(\vr);
\draw(x3)[fill=white] circle(\vr);
\draw(x4)[fill=white] circle(\vr);
\draw(x5)[fill=white] circle(\vr);
\draw(x6)[fill=black] circle(\vr);
\draw(x7)[fill=black] circle(\vr);
\draw(x8)[fill=white] circle(\vr);
\draw(x9)[fill=white] circle(\vr);
\end{scope}


\begin{scope}[xshift=2cm, yshift=0cm] 
\coordinate(x1) at (0,0);
\coordinate(x2) at (0,2);
\coordinate(x3) at (2,2);
\coordinate(x4) at (2,0);
\draw (x1) -- (x2) -- (x3) -- (x4) -- (x1) -- (x3);
\draw (x2) -- (x4);
\draw (2,0) -- (4,0);
\draw (2,6) -- (4,6);
\draw (0,2) -- (0,4);
\draw (6,2) -- (6,4);
\draw (2,2) -- (4,4);
\draw (2,4) -- (4,2);
\draw(x1)[fill=black] circle(\vr);
\draw(x2)[fill=white] circle(\vr);
\draw(x3)[fill=black] circle(\vr);
\draw(x4)[fill=black] circle(\vr);
\end{scope}
\begin{scope}[xshift=6cm, yshift=0cm] 
\coordinate(x1) at (0,0);
\coordinate(x2) at (0,2);
\coordinate(x3) at (2,2);
\coordinate(x4) at (2,0);
\draw (x1) -- (x2) -- (x3) -- (x4) -- (x1) -- (x3);
\draw (x2) -- (x4);
\draw(x1)[fill=white] circle(\vr);
\draw(x2)[fill=white] circle(\vr);
\draw(x3)[fill=white] circle(\vr);
\draw(x4)[fill=white] circle(\vr);
\end{scope}
\begin{scope}[xshift=2cm, yshift=4cm] 
\coordinate(x1) at (0,0);
\coordinate(x2) at (0,2);
\coordinate(x3) at (2,2);
\coordinate(x4) at (2,0);
\draw (x1) -- (x2) -- (x3) -- (x4) -- (x1) -- (x3);
\draw (x2) -- (x4);
\draw(x1)[fill=white] circle(\vr);
\draw(x2)[fill=black] circle(\vr);
\draw(x3)[fill=black] circle(\vr);
\draw(x4)[fill=black] circle(\vr);
\end{scope}
\begin{scope}[xshift=6cm, yshift=4cm] 
\coordinate(x1) at (0,0);
\coordinate(x2) at (0,2);
\coordinate(x3) at (2,2);
\coordinate(x4) at (2,0);
\draw (x1) -- (x2) -- (x3) -- (x4) -- (x1) -- (x3);
\draw (x2) -- (x4);
\draw(x1)[fill=white] circle(\vr);
\draw(x2)[fill=white] circle(\vr);
\draw(x3)[fill=white] circle(\vr);
\draw(x4)[fill=white] circle(\vr);
\end{scope}


\begin{scope}[xshift=12cm, yshift=0cm] 
\coordinate(x1) at (0,0);
\coordinate(x2) at (0,2);
\coordinate(x3) at (2,2);
\coordinate(x4) at (2,0);
\draw (x1) -- (x2) -- (x3) -- (x4) -- (x1) -- (x3);
\draw (x2) -- (x4);
\draw (2,0) -- (4,0);
\draw (2,6) -- (4,6);
\draw (0,2) -- (0,4);
\draw (6,2) -- (6,4);
\draw (2,2) -- (4,4);
\draw (2,4) -- (4,2);
\draw(x1)[fill=black] circle(\vr);
\draw(x2)[fill=white] circle(\vr);
\draw(x3)[fill=white] circle(\vr);
\draw(x4)[fill=black] circle(\vr);
\end{scope}
\begin{scope}[xshift=16cm, yshift=0cm] 
\coordinate(x1) at (0,0);
\coordinate(x2) at (0,2);
\coordinate(x3) at (2,2);
\coordinate(x4) at (2,0);
\draw (x1) -- (x2) -- (x3) -- (x4) -- (x1) -- (x3);
\draw (x2) -- (x4);
\draw(x1)[fill=white] circle(\vr);
\draw(x2)[fill=white] circle(\vr);
\draw(x3)[fill=white] circle(\vr);
\draw(x4)[fill=white] circle(\vr);
\end{scope}
\begin{scope}[xshift=12cm, yshift=4cm] 
\coordinate(x1) at (0,0);
\coordinate(x2) at (0,2);
\coordinate(x3) at (2,2);
\coordinate(x4) at (2,0);
\draw (x1) -- (x2) -- (x3) -- (x4) -- (x1) -- (x3);
\draw (x2) -- (x4);
\draw(x1)[fill=white] circle(\vr);
\draw(x2)[fill=black] circle(\vr);
\draw(x3)[fill=white] circle(\vr);
\draw(x4)[fill=black] circle(\vr);
\end{scope}
\begin{scope}[xshift=16cm, yshift=4cm] 
\coordinate(x1) at (0,0);
\coordinate(x2) at (0,2);
\coordinate(x3) at (2,2);
\coordinate(x4) at (2,0);
\draw (x1) -- (x2) -- (x3) -- (x4) -- (x1) -- (x3);
\draw (x2) -- (x4);
\draw(x1)[fill=white] circle(\vr);
\draw(x2)[fill=white] circle(\vr);
\draw(x3)[fill=black] circle(\vr);
\draw(x4)[fill=black] circle(\vr);
\end{scope}
\end{tikzpicture}
\caption{$\mu$-sets in $S_3^2$ and in $S_4^2$}
\label{fig:mu-S32-S42}
\end{center}
\end{figure}
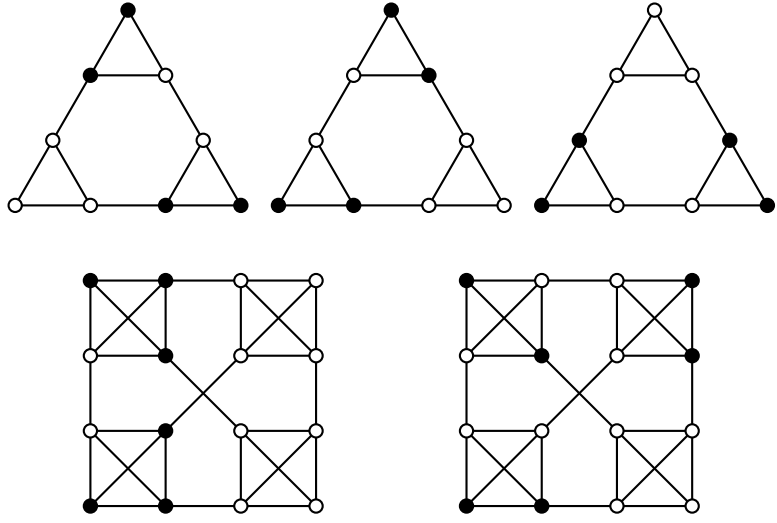

\begin{corollary}
\label{cor-gp}
If $p\geq 3$ then, 
$$\gp(S_p^2) = \begin{cases}
\frac{(p+1)^2}{4}; & p \text{ odd}\,,\\[5pt]
\frac{p(p+2)}{4}; & p \text{ even}\,,
\end{cases} 
\quad \text{and} \quad \# \gp(S_p^2) = \begin{cases}
\binom{p}{\frac{p+1}{2}}; & p \text{ odd}\,,\\[5pt]
\binom{p+1}{\frac{p+2}{2}}; & p \text{ even}\,.
\end{cases}$$
\end{corollary}

\proof
The $\mu$-sets $X_1$ and $X_2$ mentioned in the proof of Theorem~\ref{thm:mu-Sp2} are general position sets of $S_p^2$. Since $\gp(S_p^2)\le \mu(S_p^2)$, the result follows. 
\qed

\begin{theorem}\label{thm:mud-Sp2}
If $p\geq 3$, then $\mud(S_p^2) = p$ and  $\#\mud(S_p^2) = p + 1$.
\end{theorem} 

\proof
By Lemma~\ref{lem:tau lower bound}, $\mud(S_p^2)\geq p$. To prove the upper bound, we consider an arbitrary $\mud$-set $X$ of $S_p^2$. Let $X_i = X\cap V(iS_p^1)$, and let $x_i = |X_i|$ for $i\in [p]_0$. We distinguish three cases. 

If for each $i\in [p]_0$ we have $k_i\leq 1$, then $|X|\leq p$.

Assume second that there exists an index $i\in [p]_0$ such that $k_i = p$, then $k_{i'} = 0$, where $i'\in [p]_0\setminus \{i\}$. Indeed, for otherwise $00, 0i', i'0, i'j$ is the unique shortest path between $00$ and $i'j$, where $j\in [p]_0$, but this implies that the vertices $00$ and $i'j$ are not $X$-visible as $0i'\in X$. Then $|X|\leq p$. 

In the remaining case we may assume without loss of generality that $2\leq x_0\leq p-1$ and that $x_0 = \max\{x_i:\ i\in [p]_0\}$. Then there exist vertices $0i\in X$ and $0j\notin X$. In $iS_p^1$ there exists a vertex $ik\notin X$. Since $0j, 0i, i0, ik$ is the unique shortest path, the vertices $0j$ and $jk$ are not $X$-visible, hence this last case is not possible. In consequence, we have $|X|\leq p$, and we can conclude that $\mud(S_p^2) = p$.

From the above, the only possibilities that $X$ is a $\mud$-set is that $X$ contains all vertices of some $iS_p^1$, or that $X$ has exactly one vertex from each of them. In the first case, we find $\mud$-sets $V(iS_p^1)$, $i\in [p]_0$, while in the second case the only $\mud$-set is $\{ii:\ i\in [p]_0\}$. Hence $\#\mud(S_p^2) = p + 1$.
\qed

Theorem~\ref{thm:mud-Sp2} is illustrated in Fig.~\ref{fig:mud-S42} on the case of $S_4^2$. The left figure shows one of four $\mud$-sets which respectively contain sets $V(iS_4^1)$, the right figure shows the unique $\mu$-set which interests each subgraph $iS_4^1$. 

\begin{figure}[ht!]
\begin{center}
\begin{tikzpicture}[scale=0.5,style=thick,x=1cm,y=1cm]
\def\vr{5pt}

\begin{scope}[xshift=2cm, yshift=0cm] 
\coordinate(x1) at (0,0);
\coordinate(x2) at (0,2);
\coordinate(x3) at (2,2);
\coordinate(x4) at (2,0);
\draw (x1) -- (x2) -- (x3) -- (x4) -- (x1) -- (x3);
\draw (x2) -- (x4);
\draw (2,0) -- (4,0);
\draw (2,6) -- (4,6);
\draw (0,2) -- (0,4);
\draw (6,2) -- (6,4);
\draw (2,2) -- (4,4);
\draw (2,4) -- (4,2);
\draw(x1)[fill=black] circle(\vr);
\draw(x2)[fill=black] circle(\vr);
\draw(x3)[fill=black] circle(\vr);
\draw(x4)[fill=black] circle(\vr);
\end{scope}
\begin{scope}[xshift=6cm, yshift=0cm] 
\coordinate(x1) at (0,0);
\coordinate(x2) at (0,2);
\coordinate(x3) at (2,2);
\coordinate(x4) at (2,0);
\draw (x1) -- (x2) -- (x3) -- (x4) -- (x1) -- (x3);
\draw (x2) -- (x4);
\draw(x1)[fill=white] circle(\vr);
\draw(x2)[fill=white] circle(\vr);
\draw(x3)[fill=white] circle(\vr);
\draw(x4)[fill=white] circle(\vr);
\end{scope}
\begin{scope}[xshift=2cm, yshift=4cm] 
\coordinate(x1) at (0,0);
\coordinate(x2) at (0,2);
\coordinate(x3) at (2,2);
\coordinate(x4) at (2,0);
\draw (x1) -- (x2) -- (x3) -- (x4) -- (x1) -- (x3);
\draw (x2) -- (x4);
\draw(x1)[fill=white] circle(\vr);
\draw(x2)[fill=white] circle(\vr);
\draw(x3)[fill=white] circle(\vr);
\draw(x4)[fill=white] circle(\vr);
\end{scope}
\begin{scope}[xshift=6cm, yshift=4cm] 
\coordinate(x1) at (0,0);
\coordinate(x2) at (0,2);
\coordinate(x3) at (2,2);
\coordinate(x4) at (2,0);
\draw (x1) -- (x2) -- (x3) -- (x4) -- (x1) -- (x3);
\draw (x2) -- (x4);
\draw(x1)[fill=white] circle(\vr);
\draw(x2)[fill=white] circle(\vr);
\draw(x3)[fill=white] circle(\vr);
\draw(x4)[fill=white] circle(\vr);
\end{scope}


\begin{scope}[xshift=12cm, yshift=0cm] 
\coordinate(x1) at (0,0);
\coordinate(x2) at (0,2);
\coordinate(x3) at (2,2);
\coordinate(x4) at (2,0);
\draw (x1) -- (x2) -- (x3) -- (x4) -- (x1) -- (x3);
\draw (x2) -- (x4);
\draw (2,0) -- (4,0);
\draw (2,6) -- (4,6);
\draw (0,2) -- (0,4);
\draw (6,2) -- (6,4);
\draw (2,2) -- (4,4);
\draw (2,4) -- (4,2);
\draw(x1)[fill=black] circle(\vr);
\draw(x2)[fill=white] circle(\vr);
\draw(x3)[fill=white] circle(\vr);
\draw(x4)[fill=white] circle(\vr);
\end{scope}
\begin{scope}[xshift=16cm, yshift=0cm] 
\coordinate(x1) at (0,0);
\coordinate(x2) at (0,2);
\coordinate(x3) at (2,2);
\coordinate(x4) at (2,0);
\draw (x1) -- (x2) -- (x3) -- (x4) -- (x1) -- (x3);
\draw (x2) -- (x4);
\draw(x1)[fill=white] circle(\vr);
\draw(x2)[fill=white] circle(\vr);
\draw(x3)[fill=white] circle(\vr);
\draw(x4)[fill=black] circle(\vr);
\end{scope}
\begin{scope}[xshift=12cm, yshift=4cm] 
\coordinate(x1) at (0,0);
\coordinate(x2) at (0,2);
\coordinate(x3) at (2,2);
\coordinate(x4) at (2,0);
\draw (x1) -- (x2) -- (x3) -- (x4) -- (x1) -- (x3);
\draw (x2) -- (x4);
\draw(x1)[fill=white] circle(\vr);
\draw(x2)[fill=black] circle(\vr);
\draw(x3)[fill=white] circle(\vr);
\draw(x4)[fill=white] circle(\vr);
\end{scope}
\begin{scope}[xshift=16cm, yshift=4cm] 
\coordinate(x1) at (0,0);
\coordinate(x2) at (0,2);
\coordinate(x3) at (2,2);
\coordinate(x4) at (2,0);
\draw (x1) -- (x2) -- (x3) -- (x4) -- (x1) -- (x3);
\draw (x2) -- (x4);
\draw(x1)[fill=white] circle(\vr);
\draw(x2)[fill=white] circle(\vr);
\draw(x3)[fill=black] circle(\vr);
\draw(x4)[fill=white] circle(\vr);
\end{scope}
\end{tikzpicture}
\caption{$\mud$-sets in $S_4^2$}
\label{fig:mud-S42}
\end{center}
\end{figure}
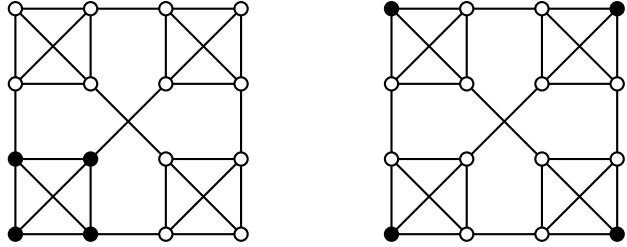

    \begin{corollary}
    \label{cor:gpd-Sp2}
    If $p\geq 3$, then $\gpd(S_p^2) = p$ and $\# \gpd(S_p^2) = 1$.
    \end{corollary}

    \begin{proof}
    It is straightforward to check that the set $\{ii:\ i\in [p]_0\}$ is a dual general position set of $S_p^2$, hence $\gpd(S_p^2)\geq p$. By the definitions of mutual-visibility and general position we have $\gpd(G)\leq \mud(G)$, in view of Theorem~\ref{thm:mud-Sp2}, hence $\gpd(S_p^2) = p$. Moreover, the set $\{ii:\ i\in [p]_0\}$ is the only largest dual general position set of $S_p^2$, hence we are done.
    \end{proof}
    \qed 

\begin{theorem}\label{th: mut-muo-sp2}
If $p\geq 3$ and $\tau\in \{\mut, \muo, \gpt, \gpo\}$, then 
$\tau(S_p^2) = p$ and $\# \tau(S_p^2) = 1$.
\end{theorem}
	
\proof
By Lemma~\ref{lem:tau lower bound} we have $\tau(S_p^2)\geq p$ for any $\tau\in \{\mut, \muo, \gpt, \gpo\}$.

We first consider the total mutual-visibility. Since $\mut(S_p^2) \le \mud(S_p^2)$, Theorem~\ref{thm:mud-Sp2} implies $\mut(S_p^2) = p$. Let $X$ be an arbitrary $\mut$-set of $S_p^2$. We will show that $X = \{ii:\ i\in [p]_0\}$. If $i\ne j$, then considering the vertices $ii$, $ij$, and $ji$, Theorem~\ref{thm:distance two} implies that $ij\notin X$ and $ji\notin X$. It follows that $X\subseteq \{ii:\ i\in [p]_0\}$. Moreover, the set  $\{ii:\ i\in [p]_0\}$ is a total mutual-visibility set, hence this set is the unique $\mut$-set of $S_p^2$.   

Consider next the outer mutual-visibility. To prove that $\muo(S_p^2)\leq p$, let $Y$ be an arbitrary $\muo$-set of $S_p^2$ and let $Y_i = Y\cap V(iS_p^1)$ for $i\in [p]_0$. If $|Y_i| \leq 1$ for each $i\in [p]_0$, then there is nothing to prove. In the rest we may hence assume that $|Y_0| \geq 2$ and that $|Y_0| = \max \{|Y_i|:\ i\in [p]_0\}$. Then there exists a vertex $0j\in Y_0$, where $j\in [p-1]$. Since $0i, 0j, j0$ is the unique shortest path between $0i$ and $j0$, we get that $0i\notin Y$ and $j0\notin Y$. But this implies that $0j$ is the unique vertex in $Y_0$, which contradicts our assumption that  $|Y_0|\geq 2$. Hence $|Y|\leq p$, and we have $\muo(S_p^2)  = p$. This argument also implies that $Y$ is a $\muo$-set if and only if $|Y_i| = 1$ and $Y\cap V(iS_p^1) = \{ii\}$ for $i\in [p]_0$. 

Similar as to the above arguments, the set $\{ii: 1\le i \le p\}$ is the only $\gpt$-set as well as the only $\gpo$-set of $S_p^2$. Hence $\gpt(S_p^2) = \gpo(S_p^2) = p$ and $\# \gpt(S_p^2) = \# \gpo(S_p^2) = 1$.
\qed

\section{Sierpi\'nski graphs $S_3^n$}
\label{sec:S3n}
In this section, we consider varities of mutual-visibility problems and general position problems on the Sierpi\'nski graphs $S_3^n$.

\begin{theorem}\label{thm:S3n-mu-total-dual}
If $n\geq 1$, then $\mut(S_3^n) = \mud(S_3^n) = 3$. Moreover, $\#\mut(S_3^n) = 1$ and $\#\mud(S_3^n) = 4$.
\end{theorem}

\proof
Clearly, $\mut(S_3^1) = \mud(S_3^1) = 3$ and by Theorems~\ref{th: mut-muo-sp2} and~\ref{thm:mud-Sp2} also $\mut(S_3^2) = \mud(S_3^2) = 3$. Hence in the remaining proof we may assume that $n\geq 3$.

By Lemma~\ref{lem:tau lower bound} we have $\mut(S_3^n)\geq 3$ so that 
$3\le \mut(S_3^n) \le \mud(S_3^n)$. To prove that $\mut(S_3^n) = \mud(S_3^n) = 3$ it thus suffices to show that $\mud(S_3^n) \leq 3$.  Let $X$ be an arbitrary $\mud$-set of $S_3^n$. We claim that 
$$X\subseteq W = V(0^{n-1}S_3^1)\cup V(1^{n-1}S_3^1)\cup V(2^{n-1}S_3^1)\,.$$ 
Suppose on the contrary that there exists a vertex $x\in X\setminus W$. Then the degree of $x$ is $3$, let $x_1, x_2, x_3$ be the neighbors of $x$, where $x_2x_3\in E(S_3^n)$. Since $X$ is a dual mutual-visibility set, either $x_1\in X$ and $x_2,x_3\notin X$, or $x_1\notin X$ and $x_2,x_3\in X$. In the first case consider a convex $P_4$ in which the edge $xx_1$ is the middle edge to get a contradiction that $X$ is a dual mutual-visibility set. In the second case we proceed similarly, expect that now the middle edge of a considered convex $P_4$ is $xx_2$. This contradiction proves the claim. 

Assume now that $i^{n-1}j\in X$, where $i,j\in [3]_0$, $i\ne j$. Then as above, considering the neighbors of $i^{n-1}j$ we infer that then $X = \{i^{n}, i^{n-1}j, i^{n-1}k\}$, where $\{i,j,k\} = [3]_0$. In this way we get the following $\mud$-sets: $\{0^{n}, 0^{n-1}1, 0^{n-1}2\}$, $\{1^{n}, 1^{n-1}0, 1^{n-1}2\}$, and $\{2^{n}, 2^{n-1}0, 2^{n-1}1\}$. So, if some vertex of the form $i^{n-1}j$ lies in $X$, then $X$ is one of these three sets. The last possibility for $X$ is then $\{0^n, 1^n, 2^n\}$ which is also a dual mutual-visibility set. We have thus proved that $\mud(S_3^n) \leq 3$ and that $\#\mud(S_3^n) = 4$.

Finally, note that any vertex of $S_p^n$ of degree $p$ is the middle vertex of a convex $P_3$. By~\cite[Lemma~5]{Tian-2024a} we get that such a vertex lies in no total mutual-visibility set. We can conclude that $\{0^n, 1^n, 2^n\}$ is the unique total mutual-visibility set. 
\qed

\begin{corollary}\label{cor:gpt-gpd-sierpinski}
If $n\geq 1$, then $\gpt(S_3^n) = \gpd(S_3^n) = 3$. Moreover, $\#\gpt(S_3^n) = 1$ and $\#\gpd(S_3^n) = 1$.
\end{corollary}

\proof
Using Theorem~\ref{thm:all}(i) and Theorem~\ref{thm:S3n-mu-total-dual} we have
$$3 = \gpt(S_3^n) \le \gpd(S_3^n) \le \mud(S_3^n) = 3\,.$$
This in turn also implies that the maximum sets from Theorem~\ref{thm:S3n-mu-total-dual} remain also for the total/dual general position case.
\qed

Next we focus on the mutual-visibility set. We first settle small cases which will serve as our basis for the later induction.

\begin{proposition}\label{prop:mu-sierpinski-small}
The following holds. 
\begin{enumerate}
\item[(i)] $\mu(S_3^2) = 4$. Moreover, the sets $\{ii, ij, kj, kk\}$, where $\{i,j,k\} = [3]_0$, are the unique $\mu$-sets of $S_3^2$.
\item[(ii)] $\mu(S_3^3) = 6$.  Moreover, if $X$ is a $\mu$-set of $S_3^3$, then either $|X \cap V(iS_3^{3})| = 2$ for every $i \in [3]_0$, or $|X \cap V(iS_3^{3})| = 3$ for exactly two $i \in [3]_0$. 
\item[(iii)] $\mu(S_3^4) = 12$. Moreover, if $X$ is a $\mu$-set of $S_3^4$, then $|X \cap V(iS_3^{3})| = 4$ for $i \in [3]_0$. In addition, if $X$ is a mutual-visibility set of $S_3^4$, then $|X \cap (V(iS_3^{3}) \cup V(jS_3^{3}))| \le 10$ for $i,j \in [3]_0$.
\end{enumerate}
\end{proposition}

\proof  
(i) This follows from Theorem~\ref{thm:mu-Sp2} and its proof, see also the top part of Fig.~\ref{fig:mu-S32-S42}. 

\medskip
(ii) The set $X = \{iji: i\neq j \text{ and } i, j \in [3]_0\}\subseteq V(S_3^3)$ is a mutual-visibility set of $S_3^3$, hence $\mu(S_3^3) \geq 6$. 

To prove that $\mu(S_3^3) \leq 6$, consider an arbitrary $\mu$-set $T$ of $S_3^3$. Since $iS_3^2$, $i\in [3]_0$, is a convex subgraph of $S_3^3$, we have $|T \cap V(iS_3^2)|\leq 4$ by (i). There is nothing to prove that $|T|\leq 6$ if $|T \cap V(iS_3^2)| \leq 2$ for $i \in [3]_0$. In this case we also have  that $|T \cap V(iS_3^{3})| = 2$, for $i \in [3]_0$. In the rest we may hence without loss of generality assume that $|T \cap V(0S_3^2)| \ge 3$.  

Assume first that $|T \cap V(0S_3^2)| = 4$. Then these vertices must be (up to symmetry) $000, 011, 021, 020$ or $011, 010, 020, 022$. In both cases we infer that no vertex from $V(1S_3^2)\cup V(2S_3^2)$ lies in $T$, hence $|T| = 4$. 

We are left with the case when $|T \cap V(0S_3^2)| = 3$. Setting $Y_1 = \{000,001,010,011\}$ we see that $|T\cap Y_1|\leq 2$. Moreover, if $|T\cap Y_1| = 2$, then $T \cap V(1S_3^2) = \emptyset$. As in the case where $|T\cap V(2S_3^2)|\geq 4$ was previously ruled out, we can conclude that $|T|\leq 6$. The same conclusion can be derived by considering the set $Y_2 = \{000,002,020,022\}$. Hence assume in the rest that $|T\cap Y_1| \leq 1$ and $|T\cap Y_2| \leq 1$ in the rest of this proof and consider the following subcases. 

Consider the case where $|T\cap Y_1| = 1$ and $|T\cap Y_2| = 1$. Assume first that $000\in T\cap Y_1$. Then since $|T \cap V(0S_3^2)| = 3$, we have $T \cap V(0S_3^3) = \{000, 012, 021\}$. Since $021, 012, 1ij$ is the unique shortest path between $021$ and each vertex $1ij\in V(1S_3^2)\setminus \{122\}$, we have  $|T \cap V(1S_3^2)| \le 1$.  Analogously, $012, 021, 2ij$ is the unique shortest path between $012$ and each vertex $2ij\in V(1S_3^2)\setminus \{211\}$, we have  $|T \cap V(2S_3^2)| \le 1$. Hence $|T| \leq 5$ in this case. 

Assume second that $000\notin T\cap Y_1$. Since we have assumed that $|T\cap V(0S_3^3)| = 3$, we get $|\{001, 010, 011\}\cap T| = 1$, $|\{002, 020, 022\}\cap T| = 1$, and $|\{012,021\}\cap T| = 1$. Assume first that $011\in T$. Since $|\{012,021\}\cap T| = 1$, we infer that $T$ can have at most one vertex in $1S_3^2$, and if so, this vertex is $122$. Now, if $122\notin T$, then $T$ has vertices only in $0S_3^2$ and in $2S_3^2$, hence by our case assumption $|T|\le 6$. And if $122\in T$, then $T\cap V(2S_3^2) = \emptyset$, and we have $|T|\le 4$. Analogously we get the conclusion if $022\in T$. Hence we are left with the cases when  $|\{001, 010\}\cap T| = 1$, $|\{002, 020\}\cap T| = 1$, and $|\{012,021\}\cap T| = 1$. Then a case analysis similar to the above leads to the required conclusion. Since we have assumed $|T\cap V(0S_3^3)| = 3$, we are now left with the case where exactly one among $T \cap Y_1$ and $T \cap Y_2$ is of cardinality 1 and the other is empty. As in the proof of the previous cases, none other than 122 and 211 can be in $T$ from $V(1S_3^2) \cup V(2S_3^2)$. Hence, $|T| \leq 5$ in this case. We can conclude that in each case $|T| \leq 6$, and therefore  $\mu(S_3^3) = 6$. In each of the cases, we also see that if $|T \cap V(0S_3^2)| = 3$, then $|T \cap V(iS_3^2)| = 3$ and $|T \cap V(jS_3^2)| = 0$, where $\{i,j\} = \{1,2\}$ (or otherwise $|T| < 6$). 

\medskip
(iii) (To help the reader follow this part of the proof, we invite the reader to use Fig.~\ref{fig:S_3^4}.) Consider $S_3^4$ and let $X = \{iiii, i012, i120, i201:\ i \in [3]_0\}$. Since $X$ is a mutual-visibility set of cardinality $12$, we get $\mu(S_3^4) \geq 12$.

To prove the reverse inequality, let $T$ be a mutual-visibility set of $S_3^4$. Since each $iS_3^3$ is a convex subgraph of $S_3^4$, combining Lemma~\ref{lem:mu-convex} with (i) we have $|T\cap V(iS_3^3)|\le 6$. We can also observe that $V(iS_3^3) \cup V(jiS_3^2) \cup V(jjiS_3^1) \cup \{jjji, jjjj\}$ induces a convex subgraph of $S_3^4$ for $i, j \in [3]_0$.

We claim that $|T \cap (V(iS_3^3) \cup V(jS_3^3))| \le 10$ for $i, j \in [3]_0$. Let $|T \cap V(2S_3^3)| = \max\{|T \cap V(iS_3^3)|:\ i\in [3]_0\}$. In view of (ii), we have $|T \cap V(2S_3^3)| \leq 6$.
If $|T \cap V(2S_3^3)|\leq 5$, there is nothing to prove and the inequality holds. If $|T \cap V(2S_3^3)| = 6$, we will show that $|T \cap V(iS_3^3)|\leq 4$ for each $i\in [2]_0$. 
Suppose not and assume, without loss of generality, that $|T \cap V(1S_3^3)| \ge 5$. Since $|T \cap V(2S_3^3)| = 6$, by $(ii)$ we have $|T\cap V(2iS_3^2)|\leq 3$ for each $i\in [3]_0$. It follows that $|T\cap (V(20S_3^2)\cup V(220S_3^1)\cup\{2220\})|\leq 5$ and $|T \cap (V(21S_3^2) \cup V(221S_3^1) \cup \{2221, 2222\})| \ge 1$. 

Next, we show that $|T \cap (V(12S_3^2) \cup V(112S_3^1) \cup \{1112, 1111\})| \ge 1$. It is straightforward to check that $|T \cap (V(21S_3^2) \cup V(221S_3^1) \cup \{2221, 2222\})| \leq 2$ if $|T\cap V(10S_3^2)|\neq 0$. By $(i)$ we known that $|T\cap V(1iS_3^2)|\leq 4$ for each $i\in [3]_0$. But if $|T\cap V(1iS_3^2)| = 4$, then $|T\cap V(1jS_3^2)| = 0$ for $j\in [3]_0\setminus \{i\}$. Hence $|T\cap V(1S_3^3)| = 4$ contradicts what we have assumed $|T \cap V(1S_3^3)| \ge 5$, so $|T\cap V(1iS_3^2)|\leq 3$ for each $i\in [3]_0$. 
Assume first that $|T\cap V(10S_3^2)| \leq 2$. Since $|T \cap (V(21S_3^2) \cup V(221S_3^1) \cup \{2221, 2222\})| \leq 2$, we have
$|T\cap (V(10S_3^2)\cup V(110S_3^1)\cup\{1110\})|\leq 4$, then our assumption implies that $|T \cap (V(12S_3^2) \cup V(112S_3^1) \cup \{1112, 1111\})| \ge 1$. 

Assume second that $|T\cap V(10S_3^2)| = 3$.  If $|T\cap (V(110S_3^1)\cup\{1110\})| \leq 1$,  then $|T \cap (V(12S_3^2) \cup V(112S_3^1) \cup \{1112, 1111\})| \ge 1$ as the assumption $|T \cap V(1S_3^3)| \ge 5$. If $|T\cap (V(110S_3^1)\cup\{1110\})| =2$, since the vertex $1100$ lies on the unique shortest path between a vertex of $T\cap \{1101, 1102, 1110\}$ and each vertex of $T\cap V(10S_3^2)$, we see that $1100\notin T$. Moreover, the vertices $1110$, $1101$, and $1102$ lie on a convex $P_3$ in $S_3^4$, then $1110\in T$ or $1101\in T$. Furthermore, if $\{1101,1102\}\subseteq T\cap V(1S_3^3)$ is the case, since $|T \cap (V(21S_3^2) \cup V(221S_3^1) \cup \{2221, 2222\})| \ge 1$, the vertex $1102$ lies on the unique shortest path between $1101$ and a vertex of $T \cap (V(21S_3^2) \cup V(221S_3^1) \cup \{2221, 2222\})$. This is a contradiction.  If $\{1110, 1102\}\subseteq T$, it is easy to verify that only $\{1000, 1002, 1012\}\subseteq T\cap V(10S_3^2)$ is the case. But the vertex $1102$ lies on the unique shortest path between $1000$ and a vertex of $T \cap (V(21S_3^2) \cup V(221S_3^1) \cup \{2221, 2222\})$, a contradiction. Therefore, $|T\cap (V(10S_3^2)\cup V(110S_3^1)\cup\{1110\})|\leq 4$, we obtain $|T \cap (V(12S_3^2) \cup V(112S_3^1) \cup \{1112, 1111\})| \ge 1$. 

Let $x \in T \cap (V(12S_3^2) \cup V(112S_3^1) \cup \{1112, 1111\})$. Since each vertex in $T \cap V(2S_3^3)$ is $T$-visible with $x$, we obtain $2111 \notin T \cap V(2S_3^3)$ and each vertex in $T \cap V(2S_3^3)$ is $T\cap V(2S_3^3)$-visible with $2111$. This implies that $(T \cap V(2S_3^3)) \cup \{2111\}$ is a mutual-visibility set of $S_3^3$ of cardinality seven, which is a contradiction.  
As a consequence, we conclude that $|T \cap (V(iS_3^3) \cup V(jS_3^3))| \le 10$ for $i, j \in [3]_0$.

Now assume that $T$ is an arbitrary $\mu$-set of $S_3^4$. We show that $|T\cap V(iS_3^3)|\le 4$ for each $i\in [3]_0$. Suppose not and we may without loss of generality let $|T \cap V(0S_3^3)| \ge 5$. Since we have proved that $|T| \ge 12$ and $|T \cap (V(iS_3^3) \cup V(jS_3^3))| \le 10$ for each $i, j \in [3]_0$, it follows that $|T\cap V(iS_3^3)|\ge 2$ for each $i \in [2]$. In fact, there are only two cases.  Either $|T\cap V(iS_3^3)|\ge 4$ for some $i\in [2]$ or $|T\cap V(iS_3^3)|\ge 3$ for each $i \in [2]$. For each $i \in [2]$, since $|T\cap V(iS_3^3)|\ge 2$, we get 
$$ |T \cap \{0000, 000i, 00i0, 00ii, 0i00, 0i0i, 0ii0, 0iii\}| \le 1\,.$$

Let $T' = T \cap \{0111, 0112, 0121, 0122, 0211, 0212, 0221, 0222\}$. 
It is also observed that $|T'|\le 2$. Assume that $|T'| = 2$, let $x_i$ be the vertex in $T'$ closer to $iS_3^3$ for each $i\in [2]$. Since $|T\cap V(iS_3^3)|\ge 2$, it is straightforward to check that $x_i\in \{0iii, 0iij\}$, where $i\in [2]$. If for some $i\in [2]$ we have $|T\cap V(iS_3^3)|\ge 4$, there is no choice for $x_j$, where $x_j\in T'$ and $j\in [2]\setminus\{i\}$, which is a contradiction. If $|T\cap V(iS_3^3)|\ge 3$ for each $i \in [2]$, then $x_i = 0iii$, which is again not possible. These two contradictions imply that $|T'|\leq 1$. 

Recall the definition of bull graph $\underline{s}B_i^n$, where $i\in [3]_0$. Since assumption $|T \cap V(0S_3^3)| \ge 5$, we have $|T \cap V(0B_i^4)| \le 2$ for $i\in [3]_0$. If $|T\cap V(0B_i^4)| = 2$ for some $i\in [2]$, then all the remaining vertices in $T \cap V(0S_3^3)$ must be from $V(0jjS_3^1)\cup \{0j0j\}$ as $|T \cap V(iS_3^3)| \ge 2$. But $|T \cap (V(0jjS_3^1)\cup \{0j0j\})| \le 2$, which is a contradiction to $|T \cap V(0S_3^3)| \ge 5$. (Notice that $\{0010, 0020\} \subseteq T \cap V(0B_0^4)$.) This contradiction implies that $|T \cap V(0B_i^4)| \le 1$ for each $i\in [2]$.

Now, let $C$ be the set of twelve vertices in $0S_3^3$ whose induced subgraph is a cycle $C_{12}$. Then $2 \le |T \cap C| \le 3$. 

Consider first the case case $|T \cap C| = 2$. Since we have proved that $|T \cap \{0000, 000i, 00i0, 00ii, 0i00, 0i0i, 0ii0, 0iii\}| \le 1$ and $|T'|\leq 1$, where $i\in [2]$, the assumption $|T \cap V(0S_3^3)| \ge 5$ implies that $T$ does not intersect $\{00ii, 0i00, 0ijj:\ i, j \in [2] \}$. It follows that $T$ intersects each of the three sets $\{0iii, 0iij, 0iji:\ i, j \in [2]\}$ and $\{0000, 000i, 00i0, 0i0i, 0ii0, 0iii\}$ for $\{i,j\} = [2]$. There are then two subcases. If $T$ intersects $\{0000, 000i, 00i0\}$, $\{0iii, 0iij, 0iji\}$, and $\{0jjj, 0jj0, 0j0j\}$ for some $i\in [2]$, we see that $T \cap(V(0S_3^3)\setminus C) \subseteq (V(0iiS_3^1) \cup \{0i0i\})$ since $|T \cap V(jS_3^3)| \ge 2$, where $j\in [3]_0$,  which is a contradiction. In the other subcase, $0010, 0020 \in T$, and $T$ intersects $\{0iii, 0iij,0iji\}$ for some $i \in [2]$. Then $T$ does not intersect $\{0012, 0021\}$. Hence $T$ intersects $\{0102, 0120\}$ and $\{0201, 0210\}$. It follows that $T \cap (V(0S_3^3)\setminus C) \subseteq V(0jjS_3^1) \cup \{0j0j\}$ since $|T \cap V(iS_3^3)| \ge 2$, which is again a contradiction. 

Consider next the case $|T \cap C| = 3$. Since $|T \cap V(0S_3^3)| \ge 5$, the set $T$ intersects at least two of the sets $\{0iii, 0iij, 0iji:\ i, j \in [2]\}$ and $\{0000, 000i, 00i0, 0i0i, 0ii0, 0iii\}$ for $i \in [2]$. Thus $T$ intersects each of the three sets $\{0011, 0012, 0021, 0022\}$ and $\{0i00, 0i0j, 0ij0, 0ijj\}$ for $i, j \in [2]$. Since $|T \cap V(iS_3^3)| \ge 2$ for $i \in [2]$, this implies that $T$ does not intersect $\{00i0, 0i00, 0i0i, 0ii0, 0iii, 0iij, 0iji, 0ijj:\ i, j \in [2] \}$. Consequently, we can conclude that $0001$ and $0002$ are in $T$, which is a contradiction, since $T$ already intersects $\{0i0j, 0ij0:\ i, j \in [2]\}$. 

Hence $\mu(S_3^4) = 12$. Moreover, if $X$ is a $\mu$-set of $S_3^4$, then $|X \cap V(iS_3^{3})| = 4$ for $i \in [3]_0$.
\qed

\begin{theorem}\label{mu-sierpinski}
If $n\geq 2$, then $\mu(S_3^n) = 3^{n-2} + 3$. Moreover, if $n \geq 4$ and $X$ is a $\mu$-set of $S_3^n$, then $|X \cap V(iS_3^{n-1})| = 3^{n-3}+1$ for $i \in [3]_0$. In addition, if $n \geq 4$ and $X$ is a mutual-visibility set of $S_3^n$, then $|X \cap (V(iS_3^{n-1}) \cup V(jS_3^{n-1}))| \le 2.3^{n-3} + 4$ for $i,j \in [3]_0$.
\end{theorem}

\proof
Proposition~\ref{prop:mu-sierpinski-small} yields the correctness of the formula for $n\le 4$. Also, note that when $n=4$, $2.3^{n-3} + 4 = 10$, so that the remaining part of the statement also follows from Proposition~\ref{prop:mu-sierpinski-small}. Hence assume in the rest that $n\ge 5$. The set 
$$X = \{s012, s120, s201:\ s \in [[3]_0]^{n-3}\} \cup \{i^n:\ i \in [3]_0\}$$ 
is a mutual-visibility set of cardinality $3^{n-2} + 3$. Therefore, $\mu(S_3^n) \geq 3^{n-2} + 3$. Also, note that $|X\cap V(iS_3^{n-1})| = 3^{n-3}+1$, for $i \in [3]_0$. 
   
We first claim that if $T$ is a mutual-visibility set of $S_3^n$, then $|T \cap (V(iS_3^{n-1}) \cup V(jS_3^{n-1}))| \le 2.3^{n-3} + 4$ for $i,j \in [3]_0$. For this it is enough to show that if $|T \cap V(iS_3^{n-1})| = 3^{n-3} + 3$ for some $i \in [3]_0$ then $|T \cap V(jS_3^{n-1})| \le 3^{n-3} + 1$ for each $j \in [3]_0 \setminus\{i\}$. Assume the contrary. Without loss of generality let $|T \cap V(2S_3^{n-1})| = 3^{n-3} + 3$ and $|T \cap V(1S_3^{n-1})| \ge 3^{n-3} + 2$. Considering $1S_3^{n-1}$, by induction hypothesis, we know that $|T \cap (V(10S_3^{n-2}) \cup V(11S_3^{n-2}))| \le 2.3^{n-4} + 4$. Since $2.3^{n-4} + 4 < 3^{n-3} + 2$, we get that $T$ intersects $12S_3^{n-2}$. This implies that $21^{n-1} \notin T$ and each vertex in $T \cap V(2S_3^{n-1})$ is $T \cap V(2S_3^{n-1})$-visible with $21^{n-1}$. Hence by adding $21^{n-1}$ to the set $T \cap V(2S_3^{n-1})$ we obtain a mutual-visibility set of $S_3^{n-1}$ with cardinality $3^{n-3} + 4$, which is a contradiction to our induction hypothesis. This contradiction proves the claim. 

If $i, j, k \in [3]_0$, then by the fact that $\mu(S_3^n) \geq 3^{n-2} + 3$ and by the above claim we obtain that 
\begin{align}
|T \cap V(iS_3^{n-1})| & = |T| - |T \cap (V(jS_3^{n-1}) \cup V(kS_3^{n-1}))| \nonumber \\
& \ge (3^{n-2} + 3) - (2 \cdot 3^{n-3} + 4) \nonumber \\
& = 3^{n-3} - 1\,. \label{eq:lower}    
\end{align}
Now suppose $T$ is a $\mu$-set of $S_3^{n}$. Then $|T| \ge 3^{n-2} + 3$.  If $|T \cap V(iS_3^{n-1})| \leq 3^{n-3} + 1$ for each $i \in [3]_0$, then we are done. Suppose now that, without loss of generality, $|T \cap V(0S_3^{n-1})| \ge 3^{n-3} + 2$. Then considering $0S_3^{n-1}$, by induction hypothesis, we know that 
$|T \cap (V(00S_3^{n-2}) \cup V(0iS_3^{n-2}))| \le 2\cdot 3^{n-4} + 4$. Since $2\cdot3^{n-4} + 4 < 3^{n-3} + 2$, we get that $T$ intersects $0iS_3^{n-2}$ for each $i \in [2]$. 

Considering $iS_3^{n-1}$, by induction hypothesis, we know that $|T \cap (V(iiS_3^{n-2}) \cup V(ijS_3^{n-2}))| \le 2\cdot3^{n-4} + 4$. For $n \ge 6$, since $2\cdot3^{n-4} + 4 < 3^{n-3} - 1$, we get that $T$ intersects $i0S_3^{n-2}$ for each $i \in [3]_0$. Hence for $n \ge 6$, we obtain that the vertices on the shortest $01^{n-1},02^{n-1}$-path are not in $T$. In addition, we obtain that every vertex in $T \cap V(0S_3^{n-1})$ is $(T \cap V(0S_3^{n-1}))$-visible with $01^{n-1}$ and $02^{n-1}$. Hence by adding $01^{n-1}$ and $02^{n-1}$ to the set $T \cap V(0S_3^{n-1})$, we get a mutual-visibility set of $S_3^{n-1}$ with cardinality at least $3^{n-3} + 4$, which is a contradiction to our induction hypothesis.

Now we are left with the case $n = 5$. In this case, we claim that $T$ intersects $V(i0S_3^3) \cup V(ii0S_3^2) \cup V(iii0S_3^1) \cup V(iiiiS_3^1)$ for each $i \in [2]$. Assume the contrary. Then by~\eqref{eq:lower} we have $|T \cap V(iS_3^5)| \ge 8$. But now we have, $|T \cap (V(ijS_3^3) \cup V(iijS_3^2) \cup V(iiijS_3^1))| \ge 8$. Then $|T \cap V(ijS_3^3)| \le 5$ since otherwise, we can form a mutual-visibility set of $S_3^3$ with cardinality seven, which is a contradiction. Also, $|T \cap (V(iijS_3^2) \cup V(iiijS_3^1))| \le 3$. Hence the only possibility is that $|T \cap V(ijS_3^3)| = 5$ and $|T \cap (V(iijS_3^2) \cup V(iiijS_3^1))| = 3$. But $|T \cap V(ijS_3^3)| = 5$ happens only if $ij00i$ and $ij00j$ are $T$ in which case, $ij00j$ is not mutually-visible with the vertices in $V(iijS_3^2) \cup V(iiijS_3^1)$. This is a contradiction. Thus $T$ intersects $V(i0S_3^3) \cup V(ii0S_3^2) \cup V(iii0S_3^1) \cup V(iiiiS_3^1)$ for both $i \in [2]$. Then again as mentioned for $n \ge 6$, we can add $01^{4}$ and $02^{4}$ to the set $T \cap V(0S_3^{4})$ to get a mutual-visibility set of $S_3^{4}$ with cardinality at least $13$, which is a contradiction. Thus if $n\geq 5$, then $\mu(S_3^n) = 3^{n-2} + 3$. Moreover, if $X$ is a $\mu$-set of $S_3^n$, then $|X \cap V(iS_3^{n-1})| = 3^{n-3}+1$ for $i \in [3]_0$. 
\qed

\begin{corollary}\label{cor：gp-Spn}
If $n\geq 2$, then $\gp(S_3^n) = 3^{n-2} + 3$.
\end{corollary}

\proof
Since $\gp(S_3^n)\leq \mu(S_3^n)$,  Theorem~\ref{mu-sierpinski} implies that $\gp(S_3^n)\leq \mu(S_3^n) = 3^{n-2} + 3$. On the other hand, the set $\{s012, s120, s201:\ s \in [[3]_0]^{n-3}\} \cup \{i^n:\ i \in [3]_0\}$ is a general position set of $S_3^n$ of cardinality $3^{n-2} + 3$ for $n\geq 3$, and we are done. 
\qed

It remains to consider the outer mutual-visibility number and the outer general position number of $S_3^n$. The sets 
\begin{align*}
X & =  \{0^k12^{n-k-1}:\ k \in [n]_0\}\quad {\rm and} \\
Y & = \{0^n, i000^kij^{n-k-4}:\ i,j \in [2], k \in [n-4]_0\}
\end{align*}
are outer general position sets with cardinality $n$ and $2n - 7$, respectively. Thus, if $n\geq 3$, then 
$$\muo(S_3^n) \ge \gpo(S_3^n) \ge \max \{n, 2n - 7\}\,.$$  

\section{Concluding remarks}

We have finished the previous section by bounding from below $\muo(S_3^n)$ and $\gpo(S_3^n)$. It remains open to determine the exact values of $\muo(S_3^n)$ and of $\gpo(S_3^n)$ for $n\geq 3$.

The Sierpi\'nski graph $S_p^2$ is isomorphic to the Sierpi\'nski product graph $K_p\otimes_f K_p$, where $f$ is the identity function. (For the definition of Sierpi\'nski product graphs see the seminal paper~\cite{kovic-2023}, cf.\ also~\cite{henning-2024}.) In~\cite{TianKlavzar-2025b} it was proved that $\min_f \gp(K_p\otimes_f K_p) = p$ and that $\max_f \gp(K_p\otimes_f K_p) = p(p-1)$. Hence Corollary~\ref{cor-gp} asserts that $S_p^2$ is somewhere in between the minimum and the maximum over all functions $f$.
  
\section*{Acknowledgments}
	
Dhanya Roy thank Cochin University of Science and Technology for providing financial support under University JRF Scheme. 
Sandi Klav\v zar was supported by the Slovenian Research and Innovation Agency ARIS (research core funding P1-0297 and projects N1-0285, N1-0355).

\end{document}